\documentclass[12 pt]{amsart}
\title{Automorphism groups of curves over arbitrary fields}
\author{Daniel Bragg}
\address{Department of Mathematics, University of Utah, Salt Lake City, UT 84112}
\email{bragg@math.utah.edu}
\usepackage{macros}
\usepackage{mathtools}

\begin{document}

\begin{abstract}
    We show that if $K$ is an arbitrary field and $G$ is a finite group then there exists a curve over $K$ with automorphism group $G$. We also give a positive solution to the weak inverse Galois problem for function fields over an arbitrary field. Finally, we reduce the inverse Galois problem for function fields over an arbitrary field $K$ to the case of $K(T)$.
\end{abstract}

\maketitle





\section{Introduction}\label{sec:intro}

The goal of this paper is to prove the following result.

\begin{theorem}\label{thm:second corollary to main theorem for curves}
    Let $K$ be a field and let $G$ be a finite group. There exists a smooth curve $C$ over $K$ such that $\Aut_K(C)=\Aut_{\overline{K}}(C_{\overline{K}})\cong G$. Moreover, we may choose $C$ to have arbitrarily large genus.
\end{theorem}

Here, by a \textit{curve} we mean a proper geometrically integral $K$-scheme of dimension 1. 
This result was previously known in the special case when $K$ is algebraically closed, due to Madden--Valentini \cite{MR705883}, and in the case when $K$ is finite, due to Rzedowski--Calder\'{o}n and Villa--Salvador \cite[Theorem 5]{MR1120718}. We also prove the following more general result, which allows the prescription of the automorphism group scheme.

\begin{theorem}\label{thm:first corollary to main theorem for curves}
    Let $K$ be a field and let $G$ be a finite \'{e}tale group scheme over $K$. There exists a smooth curve $C$ over $K$ such that $\sAut_K(C)\cong G$. Moreover, we may choose $C$ to have arbitrarily large genus.
\end{theorem}

We deduce Theorems \ref{thm:second corollary to main theorem for curves} and \ref{thm:first corollary to main theorem for curves} as a consequence of the following.

\begin{theorem}\label{thm:main theorem for curves}
    Let $K$ be a field, let $X$ be a regular curve over $K$, and let $G$ be a finite \'{e}tale group scheme over $K$. There exists a regular curve $C$ over $K$ and a finite morphism $C\to X$ such that 
    \[
        \sAut_K(C/X)=\sAut_K(C)\cong G.
    \]
    Moreover, we may choose $C$ to have arbitrarily large genus, and if $X$ is smooth, then we may choose $C$ to be smooth.
\end{theorem}

Here, $\sAut_K(C)$ denotes the automorphism group scheme of $C$ over $K$, and $\sAut_K(C/X)$ denotes the sub group scheme of automorphisms over $X$ (see \S\ref{ssec:notation}). We remark that the distinction between a regular curve and a smooth curve is only relevant over imperfect fields of positive characteristic. While constructions with smooth curves would suffice for our proof of Theorems \ref{thm:second corollary to main theorem for curves} and \ref{thm:first corollary to main theorem for curves}, we work in the slightly larger generality of regular curves in order to obtain more comprehensive results in positive characteristic.


The problem of constructing curves with prescribed automorphism groups is closely related to some classical problems in field theory. The \textit{weak inverse Galois problem} for a field $L$ is the question of whether there exists for each finite group $G$ a finite extension $F/L$ such that $\Aut(F/L)\cong G$. The \textit{(strong) inverse Galois problem} for a field $L$ requests furthermore that the extension $F/L$ be Galois. A \textit{function field} over a base field $K$ is a finitely generated field extension $L/K$ of transcendence degree 1. The association $C\mapsto k(C)$ defines an equivalence of categories between the category of regular curves over $K$ and finite morphisms and the category of function fields over $K$ with field of constants equal to $K$. Translating Theorem \ref{thm:main theorem for curves} via this equivalence, we obtain the following result, which gives in particular a positive solution to the weak inverse Galois problem for function fields over an arbitrary field.

\begin{theorem}\label{thm:inverse Galois}
    Let $K$ be a field, let $L$ be a function field over $K$, and let $K'/K$ be the field of constants of $L$. Let $G$ be a finite group. There exist finite extensions $F/L$ of arbitrarily large genus and with field of constants equal to $K'$ such that
    \[
        \Aut(F/L)=\Aut(F/K')\cong G.
    \]
\end{theorem}

This result was previously known in some special cases, which we survey below.

\begin{remark}
    Our method gives no a--priori control on the degree of the field extension $F/L$. Rather, we take as input a finite field extension $L'/L$ whose automorphism group contains $G$, and produce $F$ as a further extension $F/L'$. While we can control the degree of the extension $F/L'$, we do not have anything in particular to say about that of $L'/L$.
\end{remark}



We recall a variant of these construction problems in the context of function fields. An extension $L/K$ is \textit{regular} if $K$ is algebraically closed in $L$. The \textit{regular weak inverse Galois problem} for a field $K$ is the problem of finding for each finite group $G$ a finite extension $F/K(T)$ such that $F$ is regular over $K$ and $\Aut(F/K(T))\cong G$. The \textit{regular inverse Galois} problem requests in addition that $F/K(T)$ be Galois. Applying Theorem \ref{thm:inverse Galois} with $L=K(T)$, we obtain a positive solution of the regular weak inverse Galois problem for arbitrary fields $K$.

The strategy of the proof of Theorem \ref{thm:main theorem for curves} is to first select a regular curve $D$ equipped with a free $G$--action and a finite $G$--invariant morphism $D\to X$. The curve $D$ is obtained via a Bertini argument as a $G$--equivariant complete intersection in a projective space over $X$. As the $G$--action on $D$ is faithful, we have an inclusion $G\subset\sAut_K(D)$. The second step of the construction is to refine this curve $D$ so as to remove potential excess automorphisms. We show that there exists a $G$--invariant finite morphism $D\to\mathbf{P}^1$ such that the subgroup scheme $\sAut_K(D/\mathbf{P}^1)\subset\sAut_K(D)$ is equal to $G$. We then take a carefully chosen sequence of fiber products of curves to obtain a regular curve $C$ with $G\cong\sAut_K(C)$ and a $G$--equivariant finite morphism $C\to D$. This last step is the most involved part of the construction, and is inspired by techniques of Madan--Rosen \cite{MR1088443}, Madden--Valentini \cite{MR705883}, and Stichtenoth \cite{MR753434}. We also use Poonen's construction \cite{MR1748288} of trigonal curves over arbitrary fields with trivial automorphism group.

We also prove a result which reduces the strong inverse Galois problem for function fields over a field $K$ to the case of $K(T)$. We state our result first in a geometric context. We say that a finite morphism $C\to D$ of regular curves over a field $K$ is \textit{Galois} if the corresponding extension $k(C)/k(D)$ of function fields is Galois. A finite morphism $C\to D$ of regular curves over $K$ is \textit{geometrically Galois} if its base change to a separable closure of $K$ is Galois. The \textit{Galois group scheme} of a geometrically Galois morphism $C\to D$ is the group scheme $\sAut_K(C/D)$.

\begin{theorem}\label{thm:reduction of inverse Galois problem, curves version}
    Let $K$ be an infinite field, let $G$ be a finite \'{e}tale group scheme over $K$, and suppose that there exists a regular curve $C$ and a geometrically Galois morphism $f:C\to\mathbf{P}^1$ with Galois group scheme $G$. If $Y$ is a regular curve over $K$, then there exists a regular curve $E$ over $K$ and a geometrically Galois morphism $E\to Y$ with Galois group scheme $G$. Furthermore, if $C$ and $Y$ are smooth, we may take $E$ to be smooth.
\end{theorem}

Rephrased in terms of function fields, we have the following result, which in particular reduces the strong inverse Galois problem for all function fields $L$ over a field $K$ to the case of $L=K(T)$.

\begin{theorem}\label{thm:reduction of inverse Galois problem}
    Let $K$ be an arbitrary field, let $G$ be a finite group, and suppose that there exists a regular Galois extension $M/K(T)$ with Galois group $G$. If $L$ is a regular function field over $K$, then there exist regular Galois extensions $F/L$ with Galois group $G$ of arbitrarily large genus.
\end{theorem}

Our contribution is the infinite field case; over a finite ground field the result is due to Rzedowski--Calder\'{o}n and Villa--Salvador \cite{MR1120718} and \'{A}lvarez--Garc\'{i}a and Villa--Salvador \cite{MR2417184}. We think it is likely that, over a finite field, the stronger statement of Theorem \ref{thm:reduction of inverse Galois problem, curves version} (allowing the prescription of automorphism group schemes, rather than merely automorphism groups) should follow from the methods of loc. cit., but we have not checked this. \'{A}lvarez--Garc\'{i}a and Villa--Salvador \cite{MR2661451} also obtained some partial results in the direction of Theorem \ref{thm:reduction of inverse Galois problem} in the case of an infinite ground field.

We briefly survey some of the history of the weak inverse Galois problem for various fields and the related problem of constructing curves with prescribed automorphism groups. The weak inverse Galois problem for $\mathbf{Q}$ is known due to work of Wulf--Dieter \cite{MR719416}, E. Fried and Koll\'{a}r \cite{MR512465}, and M. Fried \cite{MR580989}. The weak inverse Galois problem for function fields over a base field $K$ was solved in the affirmative by Greenberg \cite{MR0379835} in the case when $K=\mathbf{C}$. This was extended to the case when $K$ is algebraically closed by Madden--Valentini \cite{MR705883}, and to the case when $K$ is finite by Rzedowski--Calder\'{o}n and Villa--Salvador \cite[Theorem 5]{MR1120718}. To the best of our knowledge, beyond these cases the weak inverse Galois problem for function fields over $K$, as well as the regular weak inverse Galois problem for $K$ and the statement of Theorem \ref{thm:second corollary to main theorem for curves}, were previously unknown. We remark that in fact Greenberg \cite{MR0379835} solved the \textit{strong} inverse Galois problem for function fields over $\mathbf{C}$. This was extended to function fields $L$ over an algebraically closed field $K$ by Stichtenoth \cite[Satz 1, Satz 4]{MR753434} under the additional assumption that the genus of $L$ is $\geq 2$, and shown in general by Madan--Rosen \cite{MR1088443}. In another direction, Poonen \cite{MR1748288} has given an explicit construction over an arbitrary field of a curve whose full group of automorphisms is trivial. This is the special case of Theorem \ref{thm:second corollary to main theorem for curves} when $G=1$ (but without the conclusion that $C$ may be taken to have arbitrarily large genus). Poonen's result is a key input in our constructions.





\subsection{Organization of this paper}

In \S\ref{sec:automorphisms of curves}, we collect some results on the behavior of the automorphism group of a curve under certain finite morphisms and fiber products.
In \S\ref{sec:group actions on curves} we construct curves which admit free actions of any given finite \'{e}tale group scheme, and also show that any curve admits pencils with a prescribed automorphism group. Finally, in \S\ref{sec:proofs of results} we give the proofs of the results stated in \S\ref{sec:intro}.


\subsection{Notations}\label{ssec:notation}


By a \textit{curve} over a field $K$ we mean a proper geometrically integral $K$-scheme of dimension 1. A curve over $K$ is \textit{regular} if its local ring at any point is a regular local ring. 
A curve $C$ over $K$ is \textit{smooth} if the morphism $C\to\Spec K$ is smooth. Over any field, a smooth curve is regular. The converse holds if $K$ is perfect, but may fail in general. However, if $C$ is a regular curve over a field $K$, we may always obtain a smooth model of $C$ over an extension of $K$. Indeed, if $\overline{K}$ is an algebraic closure of $K$, then $\overline{K}$ is perfect, so the normalization $\widetilde{C}_{\overline{K}}$ of the base change $C_{\overline{K}}=C\otimes_K\overline{K}$ is a smooth curve over $\overline{K}$. We also remark that if $C$ is a regular curve over $K$ and $L/K$ is a separable field extension, then the base change $C_L$ is again regular.

If $C$ is a regular curve over a field $K$, then we write $g_C$ for the genus (arithmetic or geometric) of the smooth model $\widetilde{C}_{\overline{K}}$.

A finite morphism $f:C\to D$ of regular curves over $K$ is \textit{Galois} if the corresponding extension $k(C)/k(D)$ of function fields is Galois. A \textit{Galois closure} of $C\to D$ is a finite morphism $C'\to C$ from a regular curve $C'$ over $K$ such that the composition $C'\to C\to D$ is Galois. A \textit{minimal Galois closure} of $C'\to C$ is a Galois closure such that $C'\to C$ has minimal degree.

The \textit{branch locus} of a finite morphism $f:C\to D$ of curves over $K$ is the locus of points in $C$ at which $f$ is not smooth. The \textit{ramification locus} of $f$ is the image in $D$ of the branch locus.


If $\pi:X\to S$ is a morphism of schemes, we write $\Aut_S(X)$ for the group of automorphisms $\alpha$ of $X$ such that $\pi\circ\alpha=\pi$. We may write $\Aut(X)=\Aut_S(X)$ if we think the base $S$ is clear from context. Given a morphism $f:X\to Y$, we write $\Aut_S(X/Y)$ for the subgroup of $\Aut_S(X)$ consisting of those automorphisms $\alpha$ of $X$ such that $\pi\circ\alpha=\pi$ and $f\circ\alpha=f$.

If $X$ is a projective $K$--scheme, then we write $\sAut_K(X)$ for the group scheme of $K$--automorphisms of $X$. This represents the functor on the category of $K$--schemes which associates to a $K$--scheme $T$ the group $\Aut_T(X_T)$. Given a morphism $f:X\to Y$ of projective $K$--schemes, we write $\sAut_K(X/Y)\subset\sAut_K(X)$ for the sub group scheme of $K$--automorphisms of $X$ over $Y$. This represents the functor on the category of $K$--schemes which associates to a $K$--scheme $T$ the subgroup $\Aut_T(X_T/Y_T)\subset\Aut_T(X_T)$.

\subsection{Acknowledgments}

The author thanks Bjorn Poonen for helpful correspondence. This work was supported by NSF grant \#1840190.


\section{Automorphism groups and finite morphisms of curves}\label{sec:automorphisms of curves}

In this section we will identify some conditions under which automorphisms of curves descend under finite morphisms. Let $K$ be a field. Following Madden--Valentini \cite{MR705883}, we consider the following condition on a finite morphism of curves over $K$.


\begin{definition}\label{def:condition}
    A finite morphism $f:C\to D$ of regular curves over $K$ \textit{satisfies condition} ($\ast$) if given any regular curve $V$ over $K$ and a factorization $C\to V\to D$ of $f$ such that $V\to D$ is not an isomorphism, we have
    \[
        g_V\geq\deg(V/D)^2+2(g_D-1)\deg(V/D)+2
    \]
    where $g_D$ is the genus of $D$ and $g_V$ is the genus of $V$.
\end{definition}

Before explaining some consequences of this condition, we record the following result, known as the inequality of Castelnuovo--Severi. This result is sometimes stated only for smooth curves, but in fact with our conventions holds equally for regular curves.

\begin{theorem}[Castelnuovo--Severi]
    Let $f_1:C\to D_1$ and $f_2:C\to D_2$ be finite morphisms of regular curves over $K$ of degrees $d_1$ and $d_2$. If the induced morphism $f_1\times f_2:C\to D_1\times D_2$ is birational onto its image, then we have
    \[
        g_C\leq (d_1-1)(d_2-1)+d_1g_{D_1}+d_2g_{D_2}.
    \]
\end{theorem}
\begin{proof}
    When $C,D_1,$ and $D_2$ are smooth over $K$, this is explained by Kani \cite[Proposition, pg. 26]{MR758693}. The genera of curves and the degrees of finite morphisms are unchanged upon passing to the normalization over an algebraic closure, so the regular case follows from the smooth case. We refer also to Stichtenoth \cite[Satz 1]{MR733931} for an elementary proof in the language of function fields.
    
\end{proof}

We apply the inequality of Castelnuovo--Severi to deduce the following consequence of condition $(\ast)$ (compare to Madden--Valentini \cite[Lemma 1]{MR705883}).
 

\begin{proposition}\label{prop:CSI consequence}
    Let $f:C\to D$ be a finite morphism of regular curves over $K$ which satisfies condition $(\ast)$.
    If $\alpha\in\Aut_K(C)$ is an automorphism, then there exists a unique automorphism $\beta\in\Aut_K(D)$ such that the diagram
    \[
        \begin{tikzcd}
            C\arrow{r}{\alpha}[swap]{\sim}\arrow{d}[swap]{f}&C\arrow{d}{f}\\
            D\arrow{r}{\beta}[swap]{\sim}&D
        \end{tikzcd}
    \]
    commutes.
\end{proposition}
\begin{proof}
    Let $V$ be the normalization of the image of the map $f\times (f\circ\alpha)$. For $i=1,2$ let $\pi_i:D\times D\to D$ denote the projection onto the $i$th factor. Consider the diagram
    \[
        \begin{tikzcd}
            C\arrow{r}\arrow{dr}&V\arrow{d}{p_i}\arrow{r}&D\times D\arrow{dl}{\pi_i}\\
            &D.&
        \end{tikzcd}
    \]
    Applying the inequality of Castelnuovo--Severi to the maps $p_i:V\to D$ we obtain
    \[
        g_V\leq (\deg(V/D)-1)^2+2\deg(V/D)g_D.
    \]
    But this violates the inequality in condition $(\ast)$ applied to the factorization $C\to V\xrightarrow{p_1}D$ of $f$. Therefore $p_1$ is an isomorphism. It follows that $V\to D\times D$ is a closed immersion, and $V$ is the graph of the desired automorphism $\beta$ of $D$.
\end{proof}

\begin{corollary}\label{cor:SES of automorphism groups, just sets}
    If $f:C\to D$ is a finite morphism of regular curves over $K$ which satisfies $(\ast)$, then there is a natural exact sequence
    \[
        1\to\Aut_K(C/D)\to\Aut_K(C)\to\Aut_K(D).
    \]
\end{corollary}
\begin{proof}
    We define the right hand map by sending an automorphism $\alpha\in\Aut_K(C)$ to the automorphism $\beta$ produced in Proposition \ref{prop:CSI consequence}. It is immediate that the kernel of this map is the subgroup $\Aut_K(C/D)$.
\end{proof}

The following result shows that, under some mild assumptions, the exact sequence of Corollary \ref{cor:SES of automorphism groups, just sets} extends to an exact sequence of group schemes.

\begin{corollary}\label{cor:SES of automorphism groups}
    Let $f:C\to D$ be a finite morphism of regular curves over $K$ which satisfies $(\ast)$. Suppose that either $\sAut_K(C)$ is smooth or that $D$ is smooth and $g_D\geq 2$. There is a natural exact sequence
    \[
        1\to\sAut_K(C/D)\to\sAut_K(C)\to\sAut_K(D)
    \]
    of $K$-group schemes.
\end{corollary}
\begin{proof}
    We will extend the map $\Aut_K(C)\to\Aut_K(D)$ produced in Corollary \ref{cor:SES of automorphism groups, just sets} to a map of group schemes. Let $\sHom^o_K(C,D)$ be the Hom scheme parametrizing finite morphisms $C\to D$, and consider the closed immersion $\sAut_K(D)\subset\sHom^o_K(C,D)$ of schemes given by precomposition with $f$. We claim that the map $\sAut_K(C)\to\sHom^o_K(C,D)$ given by postcomposition with $f$ factors through the closed subscheme $\sAut_K(D)$. To verify this we may assume $K$ is separably closed. If $\sAut_K(C)$ is geometrically reduced, it will further suffice to verify that this is true at the level of $K$--points, which follows from Proposition \ref{prop:CSI consequence}. If $D$ is smooth and $g_D\geq 2$, then $\sHom^o_K(C,D)$ and $\sAut_K(D)$ are \'{e}tale over $K$, and we draw the same conclusion. Thus, in either case we obtain the desired map $\sAut_K(C)\to\sAut_K(D)$. It follows from the construction that the kernel of this map is the subgroup $\sAut_K(C/D)\subset\sAut_K(C)$.
\end{proof}

\begin{remark}
    The assumption in Corollary \ref{cor:SES of automorphism groups} that $\sAut_K(C)$ is smooth is automatic if either $K$ has characteristic 0 or $C$ is smooth over $K$.
\end{remark}

We next consider fiber products of finite morphisms of curves. We record the following lemma.

\begin{lemma}\label{lem:smoothness of fiber product lemma}
    Let $f:C\to D$ and $g:E\to D$ be finite separable morphisms of regular curves over $K$ and set $F:=C\times_DE$.
    \begin{enumerate}
        \item\label{item:smooth lemma 1} $F$ is geometrically reduced, and is regular if and only if the ramification loci of $f$ and $g$ are disjoint (as subschemes of $D$).
        \item\label{item:smooth lemma 2} If $F$ is regular, then $F$ is geometrically integral (and hence is a regular curve) if and only if it is geometrically connected.
        \item\label{item:smooth lemma 3} If $C,D,$ and $E$ are smooth, then $F$ is smooth if and only if it is regular. 
    \end{enumerate}
\end{lemma}
\begin{proof}
    Considering $F$ as a closed subscheme of $C\times E$, the singular locus of $F$ is the intersection of $F$ with the product of the branch locus of $f$ and the branch locus of $g$. Thus $F$ is regular if and only if this intersection is empty, which is equivalent to the ramification loci of $f$ and $g$ being disjoint. Let $\overline{K}$ be an algebraic closure of $K$. By assumption, $C,D,$ and $E$ are geometrically integral, hence generically smooth over $K$. Furthermore, $f$ and $g$ are separable, hence generically \'{e}tale. Applying the above description of the singular locus after base change to $\overline{K}$, we see that $F_{\overline{K}}$ is generically smooth, hence generically reduced. As $F_{\overline{K}}\to C_{\overline{K}}$ is flat, we conclude that $F_{\overline{K}}$ is reduced. Therefore $F$ is geometrically reduced, which completes the verification of~\eqref{item:smooth lemma 1}.
    
    Suppose that $F$ is regular and is geometrically connected. Then $F$ is integral. The base change of $F$ to a separable closure of $K$ remains regular and connected, hence is also integral. But an integral and geometrically reduced scheme over a separably closed field is automatically geometrically integral, so $F$ is geometrically integral, which proves~\eqref{item:smooth lemma 2}.

    Finally we show~\eqref{item:smooth lemma 3}. If $F$ is regular, then by~\eqref{item:smooth lemma 1} the ramification loci of $f$ and $g$ are disjoint. This condition persists after base change to $\overline{K}$. Furthermore, as $C,D,$ and $E$ are smooth, they remain regular after base change to $\overline{K}$. Thus $F_{\overline{K}}$ is regular, so $F$ is smooth.
\end{proof}

Let $C,D,E,$ and $F$ be regular curves over $K$. Consider a commutative diagram
\begin{equation}\label{eq:a cartesian square}
    \begin{tikzcd}
        F\arrow{d}[swap]{g'}\arrow{r}{f'}&E\arrow{d}{g}\\
        C\arrow{r}{f}&D
    \end{tikzcd}
\end{equation}
over $K$, where $f,g,f',$ and $g'$ are finite separable morphisms. We consider the following condition.

\begin{definition}\label{def:condition star star}
    We say that the diagram~\eqref{eq:a cartesian square} \textit{satisfies condition} $(\ast\ast')$ if (1) it is Cartesian, and (2) there exists a minimal Galois closure $C'\to C$ of $C\to D$ such that the fiber product $C'\times_DE$ is connected.

    We say that~\eqref{eq:a cartesian square} \textit{satisfies condition} $(\ast\ast)$ if its base change to a separable closure of $K$ satisfies condition $(\ast\ast')$.
\end{definition}

The following two results give some consequences of condition $(\ast\ast)$.

\begin{proposition}\label{prop:key galois proposition}
    If the square~\eqref{eq:a cartesian square} satisfies condition $(\ast\ast)$, then pullback along $g$ induces an isomorphism 
    \[
        \sAut_K(C/D)\iso\sAut_K(F/E).
    \]
\end{proposition}
\begin{proof}
    To verify that the pullback map is an isomorphism, we may assume that $K$ is separably closed, in which case we need only prove that the map $\Aut(C/D)\to\Aut(F/E)$ is a bijection.
    Let $C'\to C$ be a minimal Galois closure of $C\to D$. Set $F'=C'\times_CF$, and consider the diagram
    \[
        \begin{tikzcd}
            F'\arrow{d}\arrow{r}&F\arrow{r}{f'}\arrow{d}{g'}&E\arrow{d}{g}\\
            C'\arrow{r}&C\arrow{r}{f}&D
        \end{tikzcd}
    \]
    which has Cartesian squares. We claim that $F'$ is a regular curve. As $F$ is regular, Lemma \ref{lem:smoothness of fiber product lemma}~\eqref{item:smooth lemma 1} implies that the ramification loci of $f$ and $g$ are disjoint. The morphism $C'\to C$ is unramified over the branch locus of $C\to D$, so the ramification loci of $C'\to D$ and $E\to D$ are disjoint. Furthermore, by assumption, the right square satisfies condition $(\ast\ast)$, so $F'$ is connected. As $K$ is separably closed, $F'$ is therefore geometrically connected. It follows from Lemma \ref{lem:smoothness of fiber product lemma} that $F'$ is a regular curve.
    
    Consider the pullback map $\Aut(C'/D)\to\Aut(F'/E)$. As $g$ is faithfully flat, this map is injective. The morphism $C'\to D$ is Galois, so we have
    \[
        \deg(C'/D)=|\Aut(C'/D)|\leq |\Aut(F'/E)|\leq\deg(F'/E).
    \]
    On the other hand, the outer rectangle is Cartesian, so we have $\deg(C'/D)=\deg(F'/E)$, and therefore these inequalities are all equalities. It follows that the pullback map $\Aut(C'/D)\to\Aut(F'/E)$ is an isomorphism and $F'\to E$ is Galois. Similarly, the pullback map $\Aut(C'/C)\to\Aut(F'/F)$ is an isomorphism.
    
    By a standard result from Galois theory, we may describe the group $\Aut(C/D)$ in terms of the Galois group $\Aut(C'/D)$ and its subgroup $\Aut(C'/C)$: namely, $\Aut(C/D)$ is canonically identified with the quotient $N_{\Aut(C'/D)}(\Aut(C'/C))/\Aut(C'/C)$. We have a similar description for the group $\Aut(F'/E)$. We have shown that the pullback map $\Aut(C'/D)\to\Aut(F'/E)$ is an isomorphism, and that it restricts to an isomorphism $\Aut(C'/C)\to\Aut(F'/F)$. It follows that pullback along $g$ induces an isomorphism between $\Aut(C/D)$ and $\Aut(F/E)$.
\end{proof}

\begin{lemma}\label{lem:Galois fact}
    Assume that the square~\eqref{eq:a cartesian square} satisfies condition $(\ast\ast)$. Suppose given a regular curve $V$ over $K$ and a factorization $F\to V\to E$ of $f'$. There exists a regular curve $W$ over $K$ and a commutative diagram
    \[
        \begin{tikzcd}
            F\arrow{r}\arrow[bend left=25]{rr}{f'}\arrow{d}[swap]{g'}&V\arrow{d}\arrow{r}&E\arrow{d}{g}\\
            C\arrow{r}\arrow[bend right=25]{rr}[swap]{f}&W\arrow{r}&D
        \end{tikzcd}
    \]
    in which both squares are Cartesian.
\end{lemma}
\begin{proof}
    Such a diagram is unique up to a unique isomorphism, so by Galois descent we may assume that $K$ is separably closed. Let $C'\to C$ be a minimal Galois closure of $C\to D$ and set $F'=C'\times_DE$. As in the proof of Proposition \ref{prop:key galois proposition}, $F'$ is also a regular curve over $K$ and $F'\to E$ is Galois. Under the Galois correspondence, the given curve $V$ corresponds to the subgroup $\Gamma_V:=\Aut(F'/V)\subset\Aut(F'/D)$. By Proposition \ref{prop:key galois proposition}, $\Gamma_V$ is the pullback of a subgroup say $\Gamma_W$ of $\Aut(C'/D)$. We set $W=C'/\Gamma_W$. We obtain a commutative diagram
    \[
        \begin{tikzcd}
            F'\arrow{d}\arrow{r}&F\arrow{d}\arrow{r}&V\arrow{d}\arrow{r}&E\arrow{d}\\
            C'\arrow{r}&C\arrow{r}&W\arrow{r}&D.
        \end{tikzcd}
    \]
    By the Galois correspondence, the degrees of $C'\to W$ and $F'\to V$ are equal. This implies that the squares are all Cartesian, and completes the proof.
\end{proof}


The following result gives some situations under which condition $(\ast\ast)$ holds.

\begin{lemma}\label{lem:fiber product lemma}
    Let $f:C\to D$ and $g:E\to D$ be finite separable morphisms of regular curves. Suppose that the ramification loci of $f$ and $g$ are disjoint and that at least one of the following conditions hold.
    \begin{enumerate}
        \item There exists a closed point of $D$ which is totally ramified under $g$.
        \item $g$ has prime degree and $\deg(C/D)<\deg(E/D)$.
    \end{enumerate}
    The fiber product $F:=C\times_DE$ is regular and geometrically integral and the resulting square~\eqref{eq:a cartesian square} satisfies condition $(\ast\ast)$. Furthermore, if $C,D,$ and $E$ are smooth, then $F$ is smooth.
\end{lemma}
\begin{proof}
    By Lemma \ref{lem:smoothness of fiber product lemma}, the assumption on ramification implies that $F$ is regular, and is smooth if $C,D,$ and $E$ are smooth. To verify the remaining claims, we may assume that $K$ is separably closed. Let $C'\to C$ be a minimal Galois closure of $C\to D$ and set $F':=C'\times_DE$. We have a commutative diagram
    \[
        \begin{tikzcd}
            F'\arrow{r}\arrow{d}&F\arrow{r}\arrow{d}&E\arrow{d}{g}\\
            C'\arrow{r}&C\arrow{r}{f}&D
        \end{tikzcd}
    \]
    with Cartesian squares. We will show that if either of the given conditions hold then $F'$ is connected. This will imply that $F$ is connected, hence a regular geometrically integral curve by Lemma \ref{lem:smoothness of fiber product lemma}, and that the square~\eqref{eq:a cartesian square} satisfies condition $(\ast\ast)$.

    Suppose that condition (1) holds. This implies that there exists a closed point of $C'$ whose preimage in $F'$ consists of a single closed point, and so $F'$ is connected. Suppose that condition (2) holds. We have that $\deg(C'/D)\leq \deg(C/D)!$, so the degrees of $C'\to D$ and $E\to D$ are coprime, and therefore $F'$ is connected.
\end{proof}



Finally, we consider the following condition on a finite morphism of curves.

\begin{definition}
    Let $\lambda$ be an integer. A finite morphism $f:C\to D$ of regular curves over $K$ is $\lambda$-\textit{incompressible} if for every regular curve $V$ over $K$ and factorization $C\to V\to D$ of $f$ such that $V\to D$ is not an isomorphism, we have $g_V\geq\lambda$.
\end{definition}

Following Stichtenoth \cite[Lemma 2]{MR753434}, we show that by taking a suitable fiber product any pencil on a regular curve can be refined to a pencil which is $\lambda$--incompressible for arbitrarily large $\lambda$.

\begin{proposition}\label{prop:incompressible lemma}
    Let $C$ be a regular curve over $K$ and let $f:C\to\mathbf{P}^1$ be a finite separable morphism. If $K$ is finite, assume furthermore that $f$ is unramified over at least two $K$-points of $\mathbf{P}^1$. For any $\lambda\in\mathbf{Z}$, there exists a finite separable morphism $\varphi:\mathbf{P}^1\to\mathbf{P}^1$ 
    such that the fiber product $D:=C\times_{\mathbf{P}^1,\varphi}\mathbf{P}^1$ is regular and geometrically integral, the resulting commutative square
    \[
        \begin{tikzcd}
            D\arrow{r}{f'}\arrow{d}[swap]{\varphi'}&\mathbf{P}^1\arrow{d}{\varphi}\\
            C\arrow{r}{f}&\mathbf{P}^1
        \end{tikzcd}
    \]
    satisfies condition $(\ast\ast)$, and the morphism $f':D\to\mathbf{P}^1$ is $\lambda$-incompressible.
\end{proposition}
\begin{proof}
For an integer $m\geq 1$, we consider the morphism
    \[
        \varphi_m:\mathbf{P}^1\to\mathbf{P}^1
    \]
given by $[x:y]\mapsto [x^m:y^m]$. Let $m$ be a positive integer such that $m\geq\lambda+1$ and which is coprime to the exponential characteristic of $K$, and let $\varphi$ be the composition of $\varphi_m$ with an appropriate automorphism of $\mathbf{P}^1$ so that $\varphi$ and $f$ have disjoint ramification loci. We claim that $\varphi$ has the desired properties. It follows from Lemma \ref{lem:fiber product lemma} that $D$ is regular and geometrically integral and that the square satisfies condition $(\ast\ast)$. It remains to verify that $f'$ is $\lambda$-incompressible. Suppose given a regular curve $V$ and a factorization $D\to V\to\mathbf{P}^1$ of $f'$ such that $V\to\mathbf{P}^1$ is not an isomorphism. We will show that $g_V\geq\lambda$. To check this, we may assume by passing to smooth models over an algebraic closure of $K$ that $C,D,$ and $V$ are smooth and that $K$ is algebraically closed. Let $C\to W\to\mathbf{P}^1$ be the factorization of $f$ produced in Lemma \ref{lem:Galois fact}, so that we have a diagram 
\[
    \begin{tikzcd}
        D\arrow{d}\arrow{r}&V\arrow{d}{\psi}\arrow{r}&\mathbf{P}^1\arrow{d}{\varphi}\\
        C\arrow{r}&W\arrow{r}&\mathbf{P}^1
    \end{tikzcd}
\]
of smooth curves with Cartesian squares. We will bound the genus of $V$ using the Riemann--Hurwitz formula for $\psi$:
\[
    2g_{V}-2=m(2g_{W}-2)+\sum_{P\in V}(e_P-1).
\]
Here, $e_P$ is the ramification index of $\psi$ at $P$. The morphism $\varphi$ is totally ramified over two points, say $Q_0,Q_1$. We assumed that the morphisms $f$ and $\varphi$ had disjoint ramification loci, so $W\to\mathbf{P}^1$ is unramified over $Q_0$ and $Q_1$. As $W\to\mathbf{P}^1$ is not an isomorphism, there are (at least) four distinct points of $W$ which lie over $Q_0$ or $Q_1$. Moreover, $\psi$ is totally ramified over each of these points, with ramification index $m$. We obtain
\begin{align*}
    g_{V}&=mg_{W}-(m-1)+\frac{1}{2}\sum_{P\in V}(e_P-1)\\
    &\geq mg_{W}-(m-1)+\frac{1}{2}(4(m-1))\\
    &\geq m-1
\end{align*}
where the last step uses the trivial bound $g_{W}\geq 0$. We assumed that $m\geq\lambda+1$, and hence conclude that $g_{V}\geq\lambda$.
\end{proof}

The following lemma gives a consequence of the $\lambda$--incompressible condition.

\begin{lemma}\label{lem:galois + incompressible}
    Let $f:C\to D$ and $g:E\to D$ be finite separable morphisms of regular curves. Assume that $f$ and $g$ have disjoint ramification loci, and that the base change of $f$ to a separable closure of $K$ is Galois and $\lambda$--incompressible for some integer $\lambda>g_E$. The fiber product $F:=C\times_DE$ is regular and geometrically integral, and the resulting square~\eqref{eq:a cartesian square} satisfies condition $(\ast\ast)$.
\end{lemma}
\begin{proof}    
     It will suffice to assume that $K$ is separably closed and show that $F$ is connected. Indeed, Lemma \ref{lem:smoothness of fiber product lemma} will then imply that $F$ is regular and geometrically integral, and as the identity $C\to C$ is up to isomorphism the unique minimal Galois closure of $C\to D$, this will show that the square~\eqref{eq:a cartesian square} satisfies condition $(\ast\ast)$.
     
     Consider the injective map
    \[
        g^*:\Aut(C/D)\hookrightarrow\Aut(F/E)
    \]
    given by pulling back along $g$. Let $W\subset F$ be a connected component and let $\Gamma\subset\Aut(C/D)$ be the subgroup consisting of those automorphisms $\alpha$ such that $g^*\alpha$ maps $W$ to itself. Via the injection $g^*$, we regard $\Gamma$ also as a subgroup of $\Aut(F/E)$. We obtain a diagram
    \[
        \begin{tikzcd}
            W\arrow{d}\arrow{r}&W/\Gamma\arrow{d}\arrow{r}&E\arrow{d}{g}\\
            C\arrow{r}&C/\Gamma\arrow{r}&D.
        \end{tikzcd}
    \]
    We claim that the map $W/\Gamma\to E$ is an isomorphism. To see this, consider two points $P,Q\in W$ which map to the same point in $E$. The automorphism group of $f$ acts transitively on fibers, so there exists an automorphism $\alpha\in\Aut(C/D)$ such that $g^*\alpha$ maps $P$ to $Q$. It follows that $g^*\alpha$ preserves $W$, and so we have $\alpha\in\Gamma$. We conclude that $P$ and $Q$ are equal in the quotient $W/\Gamma$, and therefore that $W/\Gamma\cong E$.

    From the above diagram we see that there is a finite morphism $E\cong W/\Gamma\to C/\Gamma$, so we have $g_E\geq g_{C/\Gamma}$. By assumption, $f$ is $\lambda$--incompressible for some integer $\lambda>g_E$. It must therefore be the case that the map $C/\Gamma\to D$ is an isomorphism. We conclude that $\Gamma=\Aut(C/D)$, and so $|\Aut(W/E)|\geq |\Gamma|=|\Aut(C/D)|$. On the other hand, as $f$ is Galois, we have that 
    \[
        |\Aut(W/E)|\leq \deg(W/E)\leq\deg(C/D)=|\Aut(C/D)|.
    \]
    It follows that $\deg(W/E)=\deg(C/D)$. Therefore $W=F$ and so $F$ is connected.
\end{proof}

Finally, we record the following observation for future use.

\begin{lemma}\label{lem:aut of pencil finite field case}
    Suppose that $f:C\to D$ is a finite separable morphism whose degree is prime and such that over some closed point of $D$ there are two points at which $f$ ramifies to different degrees. If $C'\to C$ is an unramified geometrically Galois morphism, then the natural map $\sAut_K(C'/C)\to\sAut_K(C'/D)$ is an isomorphism.
\end{lemma}
\begin{proof}
    To verify this, we may assume that $K$ is separably closed and that $C'\to C$ is an unramified Galois morphism, in which case it will suffice to show that the map $\Aut(C'/C)\to\Aut(C'/D)$ is an isomorphism. For this, we first observe that a map satisfying the given assumptions on ramification cannot be Galois, as the relative automorphism group of a Galois morphism acts transitively on the fibers. Hence, $C\to D$ is not Galois, and as we assumed $C'\to C$ to be unramified, neither is $C'\to D$. Set $V=C'/\Aut(C'/D)$. Then $C'\to D$ factors as $C'\to C\to V\to D$. We assumed that $C\to D$ had prime degree, so either $C=V$ or $V=D$. But $C'\to D$ is not Galois, so we cannot have $V=D$. Therefore $C=V$, and so $\Aut(C'/C)=\Aut(C'/D)$.
\end{proof}

\section{Group actions on curves}\label{sec:group actions on curves}

In this section we will prove some results concerning the existence of curves and pencils on curves equipped with certain group actions. Let $K$ be a field. The first result shows that there exist curves over $K$ whose automorphism group is large enough to include any finite group.

\begin{theorem}\label{thm:curve with a free action}
    Let $X$ be a regular curve over $K$. Let $G$ be a finite \'{e}tale group scheme over $K$. There exists a regular curve $C$ over $K$ equipped with a free $G$-action and a finite separable morphism $C\to X$ which is $G$--invariant. Furthermore, if $X$ is smooth, then we may take $C$ to be smooth.
\end{theorem}
\begin{proof}
    Let $V$ be a finite dimensional $K$--vector space equipped with a faithful $G$--action. Let $Z\subset\mathbf{P}(V)$ be the closed locus of geometric points $x\in\mathbf{P}(V)$ such that $g\cdot x=x$ for some nontrivial $g\in G$. We assume that $Z$ has codimension at least 2 in $\mathbf{P}(V)$. For instance, a suitably large tensor power of the regular representation will have this property. Let $\pi:\mathbf{P}(V)\to\mathbf{P}(V)/G$ be the quotient map. Choose a projective embedding $X\times (\mathbf{P}(V)/G)\subset\mathbf{P}^N$. If $U\subset\mathbf{P}^N$ is a complete intersection of codimension $n$, we write $C_U=(\id_X\times\pi)^{-1}((X\times(\mathbf{P}(V)/G))\cap U)$. The $G$--action on $X\times\mathbf{P}(V)$ restricts to an action on $C_U$, and we have $C_U/G=(X\times(\mathbf{P}(V)/G))\cap U$. We therefore obtain a diagram
    \[
        \begin{tikzcd}
            C_U\arrow[hook]{d}\arrow{r}{\pi_{C_U}}&C_U/G\arrow[hook]{d}\arrow[hook]{r}&U\arrow[hook]{d}\\
            X\times\mathbf{P}(V)\arrow{r}{\id_X\times\pi}&X\times (\mathbf{P}(V)/G)\arrow[hook]{r}&\mathbf{P}^N
        \end{tikzcd}
    \]
    with Cartesian squares. We will show that $C_U\to X$ has the desired properties for a suitable $U$.
    
    Suppose first that $K$ is infinite. Let $U\subset\mathbf{P}^N$ be a general linear subspace of codimension $n$. We have assumed that $Z$ has codimension at least 2 in $\mathbf{P}(V)$, so the product $X\times\pi(Z)$ has codimension at least $2$ in $X\times(\mathbf{P}(V)/G)$. Therefore $U$ does not intersect $X\times\pi(Z)$, and so the intersection $C_U/G=(X\times(\mathbf{P}(V)/G))\cap U$ is contained in the regular locus of $X\times(\mathbf{P}(V)/G)$. By Bertini's theorem (see eg. \cite{MR1612610}), $C_U/G$ is regular and geometrically integral and the map $C_U/G\to X$ is finite and separable. As $\pi$ is \'{e}tale away from $Z$, the map $C_U\to C_U/G$ is \'{e}tale and the $G$--action on $C_U$ is free. It follows that $C_U$ is regular and that the map $C_U\to X$ is finite and separable. As $C_U$ is a complete intersection of ample divisors in a smooth geometrically connected scheme, $C_U$ is geometrically connected. Combining this with the fact that $C_U/G$ is geometrically integral and that $C_U\to C_U/G$ is \'{e}tale, we conclude that $C_U$ is geometrically integral. Finally, if $X$ is smooth, then by Bertini's theorem $C_U/G$ and hence $C_U$ will also be smooth. This completes the proof in the case when $K$ is infinite.
    
    Suppose that $K$ is finite. Then regularity and smoothness are equivalent. Applying Poonen's finite field Bertini theorem \cite[Theorem 1.1]{MR2144974} we find a complete intersection $U\subset\mathbf{P}^N$ of codimension $n$ (now possibly of high degree) such that $U$ does not intersect $X\times\pi(Z)$, so that the intersection $C_U/G=(X\times(\mathbf{P}(V)/G))\cap U$ is smooth, and so that the intersection of $C_U/G$ with the subscheme $\left\{P\right\}\times(\mathbf{P}(V)/G)$ is smooth of dimension 0, where $P\in X$ is a fixed closed point (note that this intersection must be nonempty for dimension reasons). Then the map $C_U/G\to X$ is finite and separable, as is the map $C_U\to X$. As before, $C_U$ is geometrically connected, and the quotient map $C_U\to C_U/G$ is finite \'{e}tale, so $C_U$ is also smooth, and therefore geometrically integral.
\end{proof}


The following result shows, on the other hand, that there also exist curves with trivial automorphism group.

\begin{theorem}[Poonen]\label{thm:curve with no auts}
    There exists a smooth curve $Y$ over $K$ such that $\sAut_K(Y)=1$. Furthermore, we may choose $Y$ so that there exists a finite separable morphism $Y\to\mathbf{P}^1$ which is totally ramified over $\infty$.
\end{theorem}
\begin{proof}
    Poonen \cite[Theorem 1]{MR1748288} gives explicit equations (valid over any field) of curves which have trivial automorphism group scheme, and which admit a morphism to $\mathbf{P}^1$ of degree 3 which is totally ramified over $\infty$.
\end{proof}

Let $C$ be a regular curve over an infinite field $K$. We will show that $C$ admits a pencil whose automorphism group is any given finite subgroup of the automorphism group of $C$. Given a divisor $Z\subset C$, we write $\Aut_K(C;Z)$ for the subgroup of $\Aut_K(C)$ consisting of those automorphisms which map $Z$ to itself. Let $G\subset\Aut_K(C)$ be a finite subgroup. Set $D=C/G$ and let $\pi:C\to D$ be the quotient map. The following argument is inspired by a result of Maden--Rosen \cite[Proposition 2]{MR1088443}.

\begin{proposition}\label{prop:generic divisor automorphisms}
    Let $A\subset D$ be a divisor of degree $d$ over which $\pi$ is unramified, let $Z\subset D$ be a very ample divisor of degree $e$, and assume that $|G|(d-e)\geq 2g_C+3$. For a general divisor $B\in |Z|$, the divisor $A+B$ has the property that the subgroup $\Aut_K(C;\pi^{-1}(A+B))\subset\Aut_K(C)$ is equal to $G$.
\end{proposition}
\begin{proof}
    Passing to the normalization of $C$ over an algebraic closure of $K$, we may assume that $K$ is algebraically closed and $C$ is smooth over $K$. It follows from the Riemann--Hurwitz formula that a nontrivial automorphism of $C$ fixes at most $2g_C+2$ closed points.
    Thus, an automorphism of $C$ is determined by its action on any $2g_C+3$ points. Consider the subgroup $\Gamma\subset\Aut_K(C)$ consisting of those automorphisms $\alpha$ of $C$ such that $|\alpha(\pi^{-1}(A))\cap\pi^{-1}(A)|\geq |G|(d-e)$. We assume that $|G|(d-e)\geq 2g_C+3$, so this set is finite. For each $\alpha\in\Gamma$, let $V_{\alpha}\subset |Z|$ be the linear subspace consisting of those divisors $B$ such that $\alpha(\pi^{-1}(B))\subset \pi^{-1}(A+B)$. We claim that if $\alpha\notin G$ then $V_{\alpha}$ is a proper subspace of $|Z|$. Indeed, if $\alpha\notin G$, then a general divisor $B\in |Z|$ has the property that $\alpha(\pi^{-1}(B))\neq\pi^{-1}(B)$. Furthermore, a general $B\in |Z|$ also has the property that $B$ is disjoint from the divisor $\pi(\alpha^{-1}(\pi^{-1}(A)))$, and therefore $\alpha(\pi^{-1}(B))\cap\pi^{-1}(A)=\emptyset$.
    
    Let $V\subset |Z|$ be the union of the $V_{\alpha}$ as $\alpha$ ranges over the finitely many elements of $\Gamma\setminus G$. Let $B\in |Z|$ be a divisor which is not in $V$. We claim that $\Aut_K(C;\pi^{-1}(A+B))=G$. Let $\alpha$ be an automorphism of $C$ such that $\alpha\in\Aut_K(C;\pi^{-1}(A+B))$. Then $\alpha$ preserves $\pi^{-1}(A+B)$, so $|\alpha(\pi^{-1}(A))\cap\pi^{-1}(A)|\geq |G|(d-e)$, and hence we have $\alpha\in\Gamma$. We also have that $\alpha(\pi^{-1}(B))\subset \pi^{-1}(A+B)$, so $B\in V_{\alpha}$. By our choice of $B$, we conclude that $\alpha\in G$.
\end{proof}


We continue to assume that $C$ is a regular curve over an infinite field $K$. Let $G$ be a finite \'{e}tale group scheme over $K$ and suppose given a faithful action of $G$ on $C$, or equivalently an inclusion $G\subset\sAut_K(C)$. As before, we set $D=C/G$ and let $\pi:C\to D$ denote the quotient map.

\begin{theorem}\label{thm:automorphisms of pencils}
    There exists a finite separable morphism $f:C\to\mathbf{P}^1$ such that $G=\sAut_K(C/\mathbf{P}^1)$ as subgroups of $\sAut_K(C)$.
\end{theorem}
\begin{proof}
    Fix a divisor $A\subset D$ of degree $d$ over which $\pi$ is unramified and a very ample divisor $Z\subset D$ of degree $e$ such that $|G|(d-e)\geq 2g_C+3$. Let $K^s$ be a separable closure of $K$. By Proposition \ref{prop:generic divisor automorphisms}, for a general divisor $B_{K^s}\in |Z_{K^s}|$, the divisor $A_{K^s}+B_{K^s}$ has the property that the inclusion $G(K^s)\subset\Aut_{K^s}(C_{K^s};\pi^{-1}(A_{K^s}+B_{K^s}))$ is an isomorphism. As $K$ is infinite, we may find a divisor $B\in |Z|$ whose base change to $K^s$ has this property. Consider a pencil in $|Z|$ containing $B$, and let $D\to\mathbf{P}^1$ be the resulting finite morphism. The composition $C\to D\to\mathbf{P}^1$ then has the property that the inclusion $G\subset\sAut_K(C/\mathbf{P}^1)$ induces a bijection on $K^s$-points, and therefore is an isomorphism.
\end{proof}

To treat the case of a finite field, we will use the following lemma which gives the existence of pencils with certain special properties.

\begin{lemma}\label{lem:pencil over a finite field}
    Let $C$ be a smooth curve over a finite field $K$. For any integer $N$, there exists a finite separable morphism $f:C\to\mathbf{P}^1$ such that
    \begin{enumerate}
        \item\label{item:1} the degree of $f$ is a prime $\geq N$,
        \item\label{item:2} $f$ is ramified only over $\infty$, and
        \item\label{item:3} there exist two closed points of $C$ at which $f$ is ramified to different degrees.
    \end{enumerate}
\end{lemma}
\begin{proof}
    By Lang's theorem, $C$ admits an invertible sheaf of degree $1$, say $\ms L$. 
    By a result of Kedlaya \cite[Theorem 1]{MR2092132}, there exists a positive integer $d$ and sections $s_0,s_1\in\H^0(C,\ms L^d)$ which have no common zeros such that the resulting morphism $g:C\to\mathbf{P}^1$ is finite and separable and is ramified only over a single $K$--point of $\mathbf{P}^1$, which we may take to be not equal to $0$ or $\infty$. Inspired by Kedlaya's methods, we will next modify $g$ to ensure that the desired properties hold.  Let $B\subset C$ be the branching divisor of $g$, let $Z_0=V(s_1)\subset C$ be the zero divisor of $g$, let $Z_{\infty}=V(s_0)$ be the pole divisor of $g$, and fix distinct closed points $P,Q\in C$ which are not contained in $B+Z_0+Z_{\infty}$. By Riemann--Roch, for all sufficiently large integers $e$ we may find sections $t_0,t_1\in\H^0(C,\ms L^e)$ which have no common zeros such that $t_0$ vanishes on each point of $B+P+Q$ and does not vanish on any point of $Z_{\infty}$, and such that the orders of vanishing of $t_0$ at $P$ and at $Q$ are different. By Dirichlet's theorem on primes in arithmetic progressions, we may further assume that $d+pe$ is prime and is $\geq N$. Consider the sections $u_0,u_1\in\H^0(C,\ms L^{d+pe})$ defined by $u_0=s_0t_0^p$ and $u_1=s_1t_0^p+s_0t_1^p$. As $t_0$ and $t_1$ have no common zeros and $t_0$ does not vanish on any point of $Z_{\infty}=V(s_0)$, the sections $u_0$ and $u_1$ have no common zeros, and therefore define a finite morphism $f:C\to\mathbf{P}^1$. We claim that $f$ has the desired properties.
    
    The degree of $f$ is $d+pe$, which is prime and is $\geq N$, so~\eqref{item:1} holds. To check~\eqref{item:2}, suppose that $R\in C$ is a point at which $f$ ramifies. Suppose for the sake of a contradiction that $u_0$ does not vanish at $R$. Then the differential $d(u_1/u_0)$ is defined at $R$, and also vanishes at $R$. We note that $u_1/u_0=s_1/s_0+(t_1/t_0)^p$, and therefore $d(u_1/u_0)=d(s_1/s_0)$. Thus $d\left(s_1/s_0\right)$ vanishes at $R$, and so $R\in B$. But by assumption $t_0$ and hence $u_0$ vanishes at every point of $B$, which is a contradiction. We conclude that $u_0$ vanishes at $R$, and therefore $f(R)=\infty$. It remains to show that~\eqref{item:3} holds. We have that $t_0^p$ vanishes at $P$ and $Q$ to different orders. As $t_0$ does not vanish at any point of $Z_{\infty}=V(s_0)$, we have that $s_0$ does not vanish at $P$ or at $Q$. Therefore $u_0$ vanishes to different orders at $P$ and at $Q$, and so $P$ and $Q$ are ramified under $f$ to different orders.
\end{proof}

\section{Proofs of results}\label{sec:proofs of results}

In this section we give the proofs of the results stated in \S\ref{sec:intro}. We begin with the proof of Theorem \ref{thm:main theorem for curves}. This result immediately implies Theorem \ref{thm:second corollary to main theorem for curves} and Theorem \ref{thm:first corollary to main theorem for curves}. Theorem \ref{thm:inverse Galois} is a translation of Theorem \ref{thm:main theorem for curves} into the setting of function fields, and so will also follow.

Let $K$ be a field, let $X$ be a regular curve over $K$, and let $G$ be a finite \'{e}tale group scheme over $K$. Using Theorem \ref{thm:curve with a free action}, we may find a regular curve $C$ over $K$ equipped with a free $G$--action and a finite separable morphism $\pi:C\to X$ which is $G$--invariant. The following result shows that $C$ may be refined to a curve over $X$ whose automorphism group is exactly $G$, and in particular immediately implies Theorem \ref{thm:main theorem for curves}.

\begin{theorem}\label{thm:finite cover of curve}
    There exists a regular curve $E$ over $K$ equipped with a free $G$--action and a finite $G$--equivariant morphism $E\to C$ such that
    \[
        G\cong\sAut_K(E/X)=\sAut_K(E).
    \]
    Moreover, we may choose $E$ to have arbitrarily large genus, and if $C$ is smooth, then we may take $E$ to be smooth as well.
\end{theorem}

We will first give the proof in the case when $K$ is infinite.

\begin{proof}[Proof of Theorem \ref{thm:finite cover of curve} over an infinite field]
    Suppose that $K$ is infinite. By Theorem \ref{thm:automorphisms of pencils}, we may find a finite separable $G$--invariant morphism $f:C\to\mathbf{P}^1$ such that $G=\sAut_K(C/\mathbf{P}^1)$ as subgroups of $\sAut_K(C)$. Fix an integer $\lambda$. As $K$ is infinite, $f$ is necessarily unramified over infinitely many $K$--points of $\mathbf{P}^1$. We may therefore apply Proposition \ref{prop:incompressible lemma} to the morphism $f$ and $\lambda$, yielding a morphism $\varphi:\mathbf{P}^1\to\mathbf{P}^1$.
    Using Theorem \ref{thm:curve with no auts}, we choose a smooth curve $Y$ such that $\sAut_K(Y)=1$ and a finite separable morphism $g:Y\to\mathbf{P}^1$ which is totally ramified over $\infty$. After composing with an automorphism of $\mathbf{P}^1$ we may assume that the ramification locus of $g$ is disjoint from the branch locus of $\varphi$. Set $E=D\times_{\mathbf{P}^1}Y$, so that we have a diagram
    \begin{equation}\label{eq:a rectangle of curves}
        \begin{tikzcd}
            E\arrow{d}\arrow{r}{f''}&Y\arrow{d}{g}\\
            D\arrow{d}\arrow{r}{f'}&\mathbf{P}^1\arrow{d}{\varphi}\\
            C\arrow{r}{f}\arrow{d}[swap]{\pi}&\mathbf{P}^1\\
            X
        \end{tikzcd}
    \end{equation}
    of $K$--schemes with Cartesian squares. By our choice of $\varphi$ and Lemma \ref{lem:fiber product lemma}, $D$ and $E$ are regular and geometrically integral, the two squares each satisfy condition $(\ast\ast)$, and if $C$ is smooth, then so are $D$ and $E$. Pullback along $\varphi\circ g$ induces a morphism
     \begin{equation}\label{eq:pullback}
        G=\sAut_K(C/\mathbf{P}^1)\xrightarrow{(\varphi\circ g)^*}\sAut_K(E/Y).
     \end{equation}
     \begin{claim}\label{claim:1}
        The pullback map~\eqref{eq:pullback} is an isomorphism.
     \end{claim}
     \begin{proof}
            We have shown that the two squares each satisfy condition $(\ast\ast)$. The claim follows by applying Proposition \ref{prop:key galois proposition} first to $\varphi$ and then to $g$.
     \end{proof}
    
    \begin{claim}\label{claim:2}
       If $\lambda$ is sufficiently large, then the inclusion
       \[
            \sAut_K(E/Y)\subset\sAut_K(E)
       \]
       is an isomorphism.
    \end{claim}

    Before proving Claim \ref{claim:2}, we explain how Claims \ref{claim:1} and \ref{claim:2} together imply that the curve $E$ and the finite morphism $E\to C$ have the desired properties. Combining Claims \ref{claim:1} and \ref{claim:2}, we get an isomorphism
    \[
        G=\sAut_K(C/\mathbf{P}^1)\xrightarrow[(\varphi\circ g)^*]{\sim}\sAut_K(E/Y)\iso\sAut_K(E).
    \]
    Furthermore, under the pullback map $(\varphi\circ g)^*$, the subgroup $\sAut_K(C/X)\cap\sAut_K(C/\mathbf{P}^1)$ is mapped into the subgroup $\sAut_K(E/X)\cap\sAut_K(E/Y)=\sAut_K(E/X)$. But, by assumption, the subgroup $G\subset\sAut_K(C)$ is contained in $\sAut_K(C/X)$, hence we have $G=\sAut_K(C/X)\cap\sAut_K(C/\mathbf{P}^1)$, and so $\sAut_K(E/X)$ contains the image of $G$. It follows that the pullback map induces isomorphisms 
    \[
        G\iso\sAut_K(E/X)\iso\sAut_K(E).
    \]
    Finally, $f'$ is $\lambda$-incompressible, so we have the inequalities $g_{E}\geq g_D\geq\lambda$, and therefore $E$ may be chosen to have arbitrarily large genus.

    It remains only to prove Claim \ref{claim:2}.

    \begin{proof}[Proof of \ref{claim:2}]
    We will show that if $\lambda$ satisfies the inequality
        \begin{equation}\label{eq:bound on lambda}
            \lambda\geq\deg(C/\mathbf{P}^1)^2+2(g_Y-1)\deg(C/\mathbf{P}^1)+2
        \end{equation}
    then the morphism $f'':E\to Y$ satisfies condition $(\ast)$. Suppose given a factorization
        \[
            E\to V\to Y
        \]
        of $f''$ where $V$ is a regular curve and $V\to Y$ is not an isomorphism. By Lemma \ref{lem:Galois fact}, we may find a regular curve $W$ and a diagram
        \[
            \begin{tikzcd}
                E\arrow{d}\arrow{r}\arrow[bend left=25]{rr}{f''}&V\arrow{d}\arrow{r}&Y\arrow{d}{g}\\
                D\arrow{r}\arrow[bend right=25]{rr}[swap]{f'}&W\arrow{r}&\mathbf{P}^1
            \end{tikzcd}
        \]
        where both squares are Cartesian. It follows that $W\to\mathbf{P}^1$ is also not an isomorphism.
        As $f'$ is $\lambda$--incompressible, we have $g_W\geq\lambda$. We obtain the inequalities
    \begin{align*}
        g_{V}\geq g_{W}\geq \lambda&\geq\deg(C/\mathbf{P}^1)^2+2(g_Y-1)\deg(C/\mathbf{P}^1)+2\\
        &=\deg(E/Y)^2+2(g_Y-1)\deg(E/Y)+2\\
        &\geq\deg(V/Y)^2+2(g_Y-1)\deg(V/Y)+2
    \end{align*}
    where the last inequality is a consequence of the bounds $\deg(E/Y)\geq\deg(V/Y)$ and $g_Y\geq 1$ (the latter is implied by the fact that $Y$ has trivial automorphism group scheme). This shows that $f''$ satisfies $(\ast)$, as claimed.

    Now, if $\lambda$ is large enough so that the inequality~\eqref{eq:bound on lambda} holds, then by Corollary \ref{cor:SES of automorphism groups} we have an exact sequence
    \[
        1\to\sAut_K(E/Y)\to\sAut_K(E)\to\sAut_K(Y).
    \]
    By our choice of $Y$, we have $\sAut_K(Y)=1$, and therefore the inclusion $\sAut_K(E/Y)\subset\sAut_K(E)$ is an isomorphism.
    \end{proof}
\end{proof}

We now give the proof in the case of a finite ground field. The construction is generally similar to the infinite case, but requires some modifications. 
 
\begin{proof}[Proof of Theorem \ref{thm:finite cover of curve} over a finite field]
    Suppose that $K$ is finite. We first apply Lemma \ref{lem:pencil over a finite field} to $C/G$ to obtain a finite separable morphism $C/G\to\mathbf{P}^1$ of prime degree which is ramified only over $\infty$ and which is ramified to different degrees over two closed points in some fiber. We let $f:C\to\mathbf{P}^1$ denote the composition of this pencil with the quotient map $C\to C/G$. By Lemma \ref{lem:aut of pencil finite field case}, we have $G=\sAut_K(C/\mathbf{P}^1)$ as subgroups of $\sAut_K(C)$ (here we use our assumption that the $G$--action on $C$ was free, so that $C\to C/G$ is \'{e}tale). As $\mathbf{P}^1$ has at least three $K$-points over any field, $f$ is unramified over at least two $K$--points. Let $\lambda$ be an integer, let $\varphi:\mathbf{P}^1\to\mathbf{P}^1$ be a morphism obtained by applying Proposition \ref{prop:incompressible lemma} to $f$ and $\lambda$, and set $D=C\times_{\mathbf{P}^1,\varphi}\mathbf{P}^1$. Choose a smooth curve $Y$ over $K$ with trivial automorphism group scheme and apply Lemma \ref{lem:pencil over a finite field} to $Z$ to obtain a morphism $g:Y\to\mathbf{P}^1$ whose ramification locus is disjoint from that of $f':D\to\mathbf{P}^1$ and whose degree is a prime which is $>\deg(f)$. Set $E=D\times_{\mathbf{P}^1}Y$, so that we have a diagram~\eqref{eq:a rectangle of curves} with Cartesian squares. By our choice of $\varphi$ and Lemma \ref{lem:fiber product lemma}, both $D$ and $E$ are smooth geometrically integral curves over $K$ and each of the two squares satisfies condition $(\ast\ast)$. The remainder of the proof is identical to the infinite field case: by construction, $f'$ is $\lambda$-incompressible, so $g_E\geq g_D\geq\lambda$ and hence we may choose $E$ to have arbitrarily large genus, and the proofs of Claims \ref{claim:1} and \ref{claim:2} apply without change to show that $G\cong\sAut_K(E/X)$.
\end{proof}


We now prove Theorem \ref{thm:reduction of inverse Galois problem, curves version}. This result immediately implies the infinite field case of Theorem \ref{thm:reduction of inverse Galois problem}, and so this will complete the proofs of the results claimed in \S\ref{sec:intro}. We recall the notation: $K$ is an infinite field, $G$ is a finite \'{e}tale group scheme over $K$, $C$ and $Y$ are regular curves over $K$, and $f:C\to\mathbf{P}^1$ is a geometrically Galois morphism with Galois group scheme $G$.

\begin{proof}[Proof of Theorem \ref{thm:reduction of inverse Galois problem, curves version}]
    Let $\lambda$ be an integer with $\lambda>g_Y$. As $K$ is infinite, $f$ is necessarily unramified over infinitely many $K$--points of $\mathbf{P}^1$. Let $\varphi:\mathbf{P}^1\to\mathbf{P}^1$ be a morphism satisfying the conclusions of Proposition \ref{prop:incompressible lemma} applied to $f$ and $\lambda$. Let $f':D=C\times_{\mathbf{P}^1}\mathbf{P}^1\to\mathbf{P}^1$ be the base change of $f$. Choose a pencil $g:Y\to\mathbf{P}^1$ whose ramification locus is disjoint from that of $f'$. As in the proof of Theorem \ref{thm:finite cover of curve}, we set $E=D\times_{\mathbf{P}^1}Y$, so that we have a commutative diagram
    \[
        \begin{tikzcd}
            E\arrow{r}{f''}\arrow{d}&Y\arrow{d}{g}\\
            D\arrow{r}{f'}\arrow{d}&\mathbf{P}^1\arrow{d}{\varphi}\\
            C\arrow{r}{f}&\mathbf{P}^1
        \end{tikzcd}
    \]
    with Cartesian squares. By our choice of $\varphi$, $D$ is regular and geometrically integral, the lower square satisfies condition $(\ast\ast)$, and $f'$ is $\lambda$-incompressible and $\lambda>g_Y$. In fact, it follows from the construction in Proposition \ref{prop:incompressible lemma} that $f'$ remains $\lambda$-incompressible after base change to a separable closure of $K$. Furthermore, by Proposition \ref{prop:key galois proposition}, $f'$ is geometrically Galois with Galois group scheme $G$. Lemma \ref{lem:galois + incompressible} therefore implies that $E$ is regular and geometrically integral and that the upper square also satisfies condition $(\ast\ast)$. By Proposition \ref{prop:key galois proposition} again, $f''$ is geometrically Galois with Galois group scheme $G$. To verify the remaining claims, we note that by Lemma \ref{lem:smoothness of fiber product lemma}, if $C$ and $Y$ are smooth then $D$ and $E$ will also be smooth. Finally, as $f'$ is $\lambda$-incompressible, we have in particular that $g_D\geq\lambda$, and hence $g_E\geq g_D\geq\lambda$, so we may choose $E$ to have arbitrarily large genus.
\end{proof}

\bibliographystyle{plain}
\bibliography{biblio}

\end{document}